\documentclass[12pt,english,reqno]{amsart}
\usepackage{amsmath}
\usepackage{amsthm}
\usepackage{color}
\usepackage{amssymb}
\usepackage{textcase}
\usepackage{amscd}
\usepackage{pdfsync}
\usepackage{amsfonts}
\usepackage{graphics}
\usepackage{cancel}
\long\def\salta#1{\relax}  
\date{}
\newlength{\defbaselineskip}
\theoremstyle{plain}
\newtheorem{theorem}{Theorem}[section]
\newtheorem{proposition}[theorem]{Proposition}

\theoremstyle{definition}
\newtheorem{definition}[theorem]{Definition}

\newtheorem{remark}[theorem]{Remark}
\newtheorem{example}{Example}
\newtheorem{op}[theorem]{Open Problem}
\theoremstyle{remark}

\newcommand{\re}{\mathbb{R}}

\newcommand{\elle}[1]{L^{#1}(\Omega)}
\def\rn{\mathbb{R}^{N}}

\def\be{\begin{equation}}
\def\ee{\end{equation}}
\def\rife#1{(\ref{#1})}

\def\t1p0{T^{1,p}_{0}(\Omega)}

\def\m2{M^{\frac{N(p-1)}{N-1}}(\Omega)}

\def\div{\text{div}}
\def\rn{\mathbb{R}^{N}}

\def\into{\int_{\Omega}}

\def\w-1p'{W^{-1,p'}(\Omega)}
\def\pw-1p'u{L^{p'}(0,1;W^{-1,p'}(\Omega))}

\def\dys{\displaystyle}

\def\lp'n{(L^{p'}(\Omega))^{N}}

\def\mis{\text{\rm{meas}}}

\def\lio{L^{\infty}(\Omega)}
\def\huz{H^1_0 (\Omega)}

\def\bl{\color{blue}}
\def\bk{\color{black}}

\long\def\salta#1{\relax}

\author[F. Petitta]{Francesco Petitta}

\address[F. Petitta]{Dipartimento di Scienze di Base e Applicate
per l' Ingegneria, ``Sapienza", Universit\`a di Roma, Via Scarpa 16, 00161 Roma, Italy.}\email{francesco.petitta@sbai.uniroma1.it}

\thanks{The author has been partially supported by the PNPGC project, references MTM2008--03176}

\begin{document}

\begin{abstract}

In this paper we study existence and nonexistence of solutions for a Dirichlet boundary value problem whose model is
$$
\begin{cases}
\dys -\sum_{m=1}^{\infty}  a_m \Delta u^m= f&\text{in}\ \Omega \\
u=0 & \text{on}\ \partial\Omega\,,
\end{cases}
$$
where $\Omega$ is a bounded domain of $\rn$, $a_m$ is a sequence of nonnegative real numbers, and  $f$ is  in $L^q(\Omega)$, $q>\frac{N}{2}$. 
\end{abstract}

\title[On a generalized porous medium equation]{A generalized porous medium equation related to some singular quasilinear problems}
 \maketitle
\tableofcontents
\section{Introduction}

Let $\Omega$ be a bounded open set of $\rn$, $N\geq 1$,  let $a_m$ be a sequence of nonnegative real numbers, and  consider, as a model, the following Dirichlet boundary value problem  
\begin{equation}\label{1}
\begin{cases}
\dys -\sum_{m=1}^{\infty}  a_m \Delta u^m= f&\text{in}\ \Omega \\
u=0 & \text{on}\ \partial\Omega\,,
\end{cases}
\end{equation}
where $f$  is a nonnegative function in $L^q(\Omega)$, $q>\frac{N}{2}$. 

This  problem contains, at least formally,  the features of the so called \emph{filtration equation} %\footnote{con filtration equation si intende un operatore del tipo $-\Delta \Psi (u)$ con $\Psi'\geq 0$}
 or \emph{Generalized Porous Medium Equation} (GPME). The literature about 
GPME is huge, and we refer to \cite{va} (and the references therein) for a wide account
  on this topic as well as the main possible applications.  
Only recall that a filtration equation is, roughly speaking,
 a problem involving an operator of the type $\Delta \psi (u)$ with $\psi'(s)>0$.

Although the general case could be definitely taken into account, for the sake of exposition we will only consider positive data $f$. The reason is twofold; it allows to deal with nonnegative solutions and, on the other hand, it is the right framework for most of the concrete cases.

However, our first motivation in the study of this problem was a purely mathematical one. More precisely, we started from a \emph{formal connection}, in the case $a_m=\frac1m$, with the following elliptic boundary value problem %\footnote{a questo punto...formal connection....da capire....senno' si toglie e {\it nun more nessuno!!} :) }
\begin{equation}\label{2}
\begin{cases}
\dys-\div\left(\frac{\nabla v}{1-v}\right)=f & \mbox{in}\,\Omega,\\
\hfill v=0 \hfill & \mbox{on}\,\partial\Omega,
\end{cases}
\end{equation}
studied in \cite{o} (see Example \ref{ese1} below). 

\salta{
\medskip

In fact, let $a_m = \frac{1}{m(2m-1)}$. Let  $f\in \lio$ and suppose we have a solution $u$ to problem \rife{1}. Suppose moreover that $0\leq u< 1$. Then,  
let us \emph{formally} multiply  equation in \rife{1} by $u$. By standard computations we easily get
$$
\sum_{m=0}^\infty \frac{1}{{(2m+1)
^2}}\into {|\nabla u^{2m+1}|^2} \leq \|f\|_\infty.
$$ 
On the other hand let us multiply the equation in \rife{2} by $v$ ($0\leq v<1$ ) and let us use the standard expansion $\frac{v}{1-v}=\sum_{m=1}^{\infty} v^m $. We obtain 
$$
\into |\nabla v|^2 +\sum_{m=1}^{\infty}\into v^m |\nabla v|^2 \leq \|f\|_\infty,
$$
which implies by easy manipulations,
$$
\sum_{m=0}^\infty \frac{1}{{(2m+1)
^2}}\into {|\nabla v^{2m+1}|^2} \leq \|f\|_\infty.
$$ 

So that, the two problems turn out to share some very peculiar estimates. In \cite{blop} the authors proved that problem \rife{2} does not admit a solution if the datum has a size too large.

As a consequence of our results (see theorem \ref{noex} later) we will obtain an elementary proof of the following

\begin{theorem} Let $f\in \lio$. \label{cong}
Then, there exists a positive number $\Lambda$, such that problem 
\begin{equation}
\begin{cases}
\dys -\sum_{m=1}^{\infty}  \frac{1}{m(2m-1)} \Delta u^m= f &\text{in}\ \Omega \\
u=0 &\text{on}\ \partial\Omega,
\end{cases}
\end{equation}
does not admit any weak solution for $\|f\|_{\lio} >\Lambda$.
\end{theorem} 
}

Notice that the fact that the solution should be (strictly) less than $1$ for problem \rife{2} is, let us say,   \emph{structured}
in the problem because of the presence of the singularity at $v=1$. On the other hand,  in problem \rife{1} this singularity is, in some sense,  hidden by 
the presence of the \emph{infinite sum of nonlinear slow diffusions}. Anyway the singular bound can still be there as 
we will  show later (see Proposition \ref{pro} below).  Just recall that the term \emph{slow diffusion} comes from the fact that, for instance in the model case $a_m\equiv 1$,  by trivial computations, we are lead to the following degenerate elliptic operators
$$
-\div(m u^{m-1}\nabla u),
$$
and the \emph{degenerate coefficient} (and then the diffusion) near  $u=0$ becomes \emph{smaller} for $m>1$ (and smaller and smaller as $m$ grows).

\medskip

%\bl{\bf  pseudo motivazione fisica ?questo e' il punto.....}\bk

\medskip

We will prove both existence and nonexistence results for weak solutions to problem \rife{1} depending on the character of the sequence $a_m$ and in particular on the radius of convergence $\sigma$ of the associated power series, defined by 
\be\label{rc}
\limsup_{m\to+\infty} \sqrt[m]{a_{m}}=\frac{1}{\sigma},
\ee
where also the extreme cases $\sigma=0$ and $\sigma=+\infty$ are allowed when 
the above limit is respectively $+\infty$ or $0$. As it will be clear in a while, to have the radius of 
convergence $\sigma$ defined as in \rife{rc} would formally correspond to have 
a singularity at $v=\sigma$ in an associated singular elliptic problem of the type \rife{2}.

We will also assume, and this will be essential in our analysis, that infinitely 
many $a_m$ are different from zero, and that, in particular, $a_1\neq 0$.

Roughly speaking, we prove nonexistence of a solution in $\huz$ if 
\be\label{rog}
\sum_{m=1}^\infty a_m \sigma^m<+\infty,
\ee
and the size of the datum is large. 
In fact, this result can be interpreted as a \emph{formal counterpart} of the results in \cite{o}, 
where the author proved,  that, if $h(s)\geq\alpha>0$ is a continuous function  in $[0,\sigma)$, then problem 
$$
\begin{cases}
 -\div( h(v)\nabla v)= f &\text{in}\ \Omega \\
v=0 & \text{on}\  \partial\Omega\,,
\end{cases}
$$
admits a weak solution for any data if 
$$
H(\sigma)=\int_0^\sigma  h(s)ds=+\infty\,.$$

On the other hand if \rife{rog} is not in force  we can find a solution for any data.  
We will also see that a solution to problem \rife{1}  does exist, no matter of \rife{rog}, 
if the size of the datum is small enough. Finally notice that the assumption  $h(s)\geq\alpha>0$ in \cite{o} 
has $a_1\neq 0$ as a counterpart in our case.

\subsection{Two guide examples}

To better understand the problem let us show what happens in two concrete examples. 

\begin{example}\label{ese1}
Consider the following singular elliptic problem 
\begin{equation}\label{eqese1}
\begin{cases}
\dys -\div\left(\frac{\nabla v}{1-v}\right)=f & \mbox{in}\,\Omega,\\
\hfill v=0 \hfill & \mbox{on}\,\partial\Omega.
\end{cases}
\end{equation}
In \cite{o} it is proved that, for $f\in L^q (\Omega)$, $q>\frac{N}{2}$, there exists a constant 
$0<\theta<1$ and a nonnegative function $v\in \huz$ such that $v$ solves \rife{eqese1} in the weak sense
 and $0\leq v<\theta$ a.e. on $\Omega$. In other words the solution is not affected by the presence of the singularity.  
So that we are allowed to expand in power series the term $\frac{1}{1-v}$ in the weak formulation of \rife{eqese1} to 
get 
$$
\sum_{m=1}^{\infty} \frac{1}{m} \into \nabla v^m\cdot\nabla \varphi=\into f\varphi,
$$
for any $\varphi\in \huz$. Notice that the radius of convergence of the power series associated to $a_m=\frac{1}{m}$  is $1$ and, obviously,   
$$
\sum_{m=1}^{\infty} a_m=+\infty.
$$
\end{example}

\begin{example}\label{ese2}
Let us now turn our attention to nonnegative solutions for the following singular elliptic problem
\begin{equation}\label{eqese2}
\begin{cases}
-\div((1-\log(1-v)){\nabla v})=f & \mbox{in}\,\Omega,\\
\hfill v=0 \hfill & \mbox{on}\,\partial\Omega.
\end{cases}
\end{equation}
In \cite{o} it is proved that a weak solution to problem \rife{eqese2} does exist if the size of the  datum  $f$ is small. Let us consider the formal power series of the logarithm for $0<v<1$. We are led to 
\begin{equation}\label{loga}
\begin{cases}
\dys -\Delta v -\sum_{m=2}^{\infty}  \frac{1}{m(m-1)} \Delta v^m= f &\text{in}\ \Omega \\
v=0 &\text{on}\ \partial\Omega.
\end{cases}
\end{equation}
Again, if we define $a_1=1$ and $a_m=\frac{1}{m(m-1)}$ for $m>1$, then the radius of convergence of the associated power series is $1$ but now 
$$
\sum_{m=1}^{\infty} a_m <+\infty.
$$
We will prove, in Section \ref{no}, that problem \rife{loga} does not admit any weak solution if the size of the datum $f$ (i.e. $\|f\|_{L^q (\Omega)}$) is large enough.
\end{example}

\medskip

The plan of the paper is as follows: in the next section we provide our general assumptions
 and we state the main results by also showing some useful a priori properties of the  
 solutions we are concerned with. In Section \ref{no} we prove the nonexistence result under assumption \rife{rog} if the size of the datum is large enough. Section \ref{sec4} will be devoted to the proof of  existence of a
   solution if either \rife{rog} is not in force or the size of datum $f$ is small enough. Finally, in the last section we study the behavior of the approximating sequence of solutions when existence (for the limit problem) does not hold.

\medskip

{\bf Notation.} In this paper the symbol  $C$, if not explicitly stressed, will denote any positive constant which depends on the data of the problem but never on the solution itself. Moreover, the value of $C$ may change line by line.

\section{General assumptions and statement of the main results}

Let $\Omega$  be a bounded open set of $\rn$, $N\geq 1$. Moreover, 
let $A :\rn \mapsto \mathcal{M}^{N\times N}$ be a symmetric matrix satisfying the following
 standard assumptions: there exist two positive constants $0<\alpha\leq\beta$ such that 
\begin{equation}\label{a1}
 \alpha|\xi|^2\leq A (x)\xi\cdot\xi, \ \ \text{a.e. on} \ \Omega, \forall \xi\in\rn, \xi\neq 0\,, 
\end{equation}
and 
\begin{equation}\label{a2}
|A (x)|\leq \beta, \ \ \text{a.e. on} \ \Omega\,.
\end{equation}
We then define 
$$
A_m (x) :=a_m A(x),
$$
where  $\{a_m\}$ is a sequence  of nonnegative real numbers with $a_1\neq 0$.
 We also assume that $a_m\neq 0 $ for infinitely many $m>1$. Occasionally, without loss 
of generality, we shall be allowed to suppose that \emph{all} of the  $a_m$ are strictly positive.
\bk 

We are interested  in the following Dirichlet boundary value problem  
\begin{equation}\label{main}
\begin{cases}
\dys -\sum_{m=1}^{\infty}  \div (A_m (x) \nabla u^m)= f &\text{in}\ \Omega \\
u=0 & \text{on}\ \partial \Omega,
\end{cases}
\end{equation}
with $f$  a nonnegative function in $L^q(\Omega)$, $q>\frac{N}{2}$.  Notice that, if $A (x)= I$, we are led  the model problem \rife{1}.

Let us state what we mean for weak solution to problem \rife{main}. As suggested by Example \ref{ese1}, the natural way to define it is the following

\begin{definition}\label{def} 
We say that a nonnegative function $u\in \huz\cap\lio$ is a weak solution for problem \rife{main} if 
\begin{equation}\label{defe}
\sum_{m=1}^{\infty} \into A_m (x)\nabla u^m\cdot\nabla \varphi=\into f\varphi,
\end{equation}
for any $\varphi\in \huz$. 
\end{definition}

\begin{remark}
Let us observe that, if $m>1$ and $u\in\huz\cap\lio$, then also $u^m$ is in $\huz$, so that all the terms in \rife{defe} are well defined.  The request of boundedness on $u$ is natural since we deal with data $f$ in $L^q (\Omega)$, $q>\frac{N}{2}$ (see \cite{s}). 

Also notice that, if $0\leq u\leq\gamma< \sigma$ (where $\sigma$ is as in \rife{rc}), then the left hand side of \rife{defe} is always finite. Indeed,  we have
$$
\begin{array}{l}
\dys\left|\sum_{m=1}^{\infty}  \into A_m (x) \nabla u^m\cdot\nabla \varphi\right|\\
\dys =\left|\sum_{m=1}^{\infty}  m\into A_m (x) u^{m-1}\nabla u\cdot\nabla \varphi\right| \leq \beta \sum_{m=1}^{\infty} a_m m\gamma^{m-1}\into|\nabla u\cdot\nabla \varphi|\\\\
\dys\leq C\sum_{m=1}^{\infty} a_m m \gamma^{m-1}<+\infty.
\end{array}
$$
\bk
\end{remark}

On the other hand, as a peculiarity of these problem we notice that a solution  in $\huz$ is not expected to exist with $\mis(\{u>\sigma\})>0$.  
More precisely we have

\begin{proposition}\label{pro}
Let $f\in L^q (\Omega)$, $q>\frac{N}{2}$ be a nonnegative function. Then any weak solution  of \rife{main} satisfies
$$
0\leq u\leq \sigma\ \ \text{a.e. on}\ \Omega\,.
$$ 
\end{proposition}

\bk 

Thanks to the previous result we have a precise picture of the situation which is summarized 
in the result below 
\begin{theorem}\label{th}
Let $0\leq \sigma\leq \infty$.  Then a weak solution to problem \rife{main} does exist for any $f\in L^q(\Omega)$, $q>\frac{N}{2}$, if and only if 
\begin{equation}\label{cond}
\sum_{m=1}^\infty a_m \sigma^m =+\infty
\end{equation}
\end{theorem}\bk

\medskip
\begin{remark}
 Some considerations are in order to be done about the two extreme cases 
 $\sigma=0$ and $\sigma=+\infty$. Proposition \ref{pro} allows us to say,
  in the particular case $\sigma=0$ (where \rife{cond} always fails), that any weak solution 
  of problem \rife{main} turns out to be $0$ a.e. on $\Omega$. So that, no nontrivial solutions
   are allowed in this case.  On the other hand,  in the limit case $\sigma=+\infty$ (where \rife{cond} is trivially satisfied) the situation is simpler and,  as we will see, the existence of a solution can be proved. 
\end{remark}
\medskip

  The proof of Theorem \ref{th} will be a consequence of  the two results below. In order to state the nonexistence  result it will be useful to consider an slightly different problem. That is, for fixed $\lambda>0$, we consider
  \begin{equation}\label{mainau}
\begin{cases}
\dys -\sum_{m=1}^{\infty}  \div (A_m (x) \nabla u^m)= \lambda f &\text{in}\ \Omega \\
u=0 & \text{on}\ \partial \Omega,
\end{cases}
\end{equation}
with $f$  a nonnegative function in $L^q(\Omega)$, $q>\frac{N}{2}$ such that $f\not\equiv 0$.
\begin{theorem}\label{noex}
Let 
$$
\sum_{m=1}^{\infty}  a_m \sigma^m<+\infty. 
$$
Then there exists a positive number ${\Lambda_{f}}$, such that problem \rife{mainau} 
does not admit any weak solution if $\lambda>\Lambda_{f}$.

\begin{theorem}\label{esiste}
Let $f\in L^q (\Omega) $, $q>\frac{N}{2}$, and $\{a_m\}$ such that condition \rife{cond} is in force  with $0<\sigma\leq \infty$. Then there exists a weak solution for problem \rife{main}.
\end{theorem}

\end{theorem}

\section{Nonexistence for large data}\label{no}

Let us start this section by proving Proposition \ref{pro}. 
\begin{proof}[Proof of Proposition \ref{pro}]\bl
 \bk
Let $u$ be a weak solution to problem \rife{main}. The solution is nonnegative by definition. We want to show that $u\leq \sigma$. 
\salta{
First we multiply by $G_k (u)$, we drop all the terms but the first one  to get 
 $$
 a_1 \into |\nabla G_k(u)|^2 \leq \into f G_k (u),
 $$
which leads, by Stampacchia's theorem (see \cite{s}), to 
$$
\|u\|_{\lio}\leq C\|f\|_{L^q(\Omega)}.
$$
}

We use $u$ \rife{defe}, we drop all the positive terms but the first one to get, 
using \rife{a1},   H\"older  and Sobolev inequalities, 
$$
\|u\|_{\huz}\leq C \|f\|_{L^q(\Omega)},$$ 
where the constant $C$ only depends on $\alpha, a_1$ and $|\Omega|$. 

 On the other hand, again by choosing $u$ as test function, we now  drop 
 all the positive terms but the $n$-th, to get, reasoning as before and using  the previous inequality
$$
a_n n\into u^{n-1}|\nabla u|^2 \leq C \|f\|_{L^q(\Omega)}^2.\bk
$$
Now, for fixed $M>\sigma$,  we have 
$$
\into u^{n-1}|\nabla u|^2\geq M^{n-1} \int_{\{u\geq M\}} |\nabla u |^2, 
$$
so that, for infinitely many $n$, we have  
$$
\int_{\{u\geq  M\}} |\nabla u |^2\leq C\frac{\|f\|_{L^q(\Omega)}^2}{n a_n M^n},
$$
which implies, using the definition of $\sigma$ and taking the liminf as $n$ goes to infinity, $u\leq M$, a.e. Due to the arbitrary choice of $M$ we get $0\leq u\leq \sigma$ a.e. on $\Omega$. 

\end{proof}
\bk

Let us turn now to our nonexistence result. We are in position to prove Theorem \ref{noex}. 
We define
\begin{equation}\label{kappa}
\sum_{m=1}^{\infty}  a_m \sigma^m:=K<\infty. 
\end{equation}

\medskip

Let us also  recall  that, for $f\in L^q(\Omega)$ with $q>\frac N2$, and $f\not\equiv 0$, is possible to define (see for instance \cite{def} for further details) the first
positive eigenvalue $\lambda_1(A,f)$ of the weighted  eigenvalue boundary value
problem
\begin{equation*}
\begin{cases}
-\div (A(x)\nabla \varphi) = \lambda\, f\, \varphi & \mbox{in }\Omega,
\\
\hfill \varphi=0 \hfill &\mbox{on } \partial \Omega,
\end{cases}
\end{equation*}
as
\be\label{la}
\dys\lambda_1(A,f)\, := \inf_{\stackrel{v \in H^{1}_{0}(\Omega)}{ v\neq 0} }\frac{\int_\Omega\,A(x)\nabla
v\cdot \nabla v}{\int_{\Omega}\,f\,v^{2}}\, .
\ee
Moreover, the infimum is attained by a positive eigenfunction $\varphi_1 (A,f)$ which also solves the associated Euler-Lagrange equation.

We will use $\lambda_1 (A,f)$  and $\varphi_1 (A,f)$ (or, to simplify the notation,
 $\lambda_1 (f)$  and $\varphi_1 (f)$)  to prove \emph{one side} of Theorem \ref{th} (i.e. Theorem \ref{noex}).

\begin{proof}[Proof of Theorem \ref{noex}]  
Suppose by contradiction that a solution $u$ to problem \rife{mainau} does exist for any $\lambda>0$.
 We take $\varphi_1 (f)$  as test function for \rife{mainau} 
and we get
$$
\sum_{m=1}^{\infty}  a_m\into A(x)\nabla u^m \cdot\nabla \varphi_1(f) =\lambda \into f\varphi_1(f), 
$$
so that by definition, recalling that $A(x)$ is symmetric, we have
$$
\sum_{m=1}^{\infty}  a_m\lambda_1 (f) \into  f \varphi_1(f) u^m =\lambda \into f \varphi_1(f). 
$$

Since  $0\leq u\leq \sigma$ by Proposition \ref{pro}, we can write
$$
\begin{array}{l}
\dys \lambda \into f\varphi_1(f) = \lambda_1 (f)\sum_{m=1}^{\infty}  \into a_m u^mf  \varphi_1\\\\
\dys \leq  \lambda_1(f)\sum_{m=1}^{\infty}  \into a_m \sigma^m  f \varphi_1(f)= K\lambda_1 (f) \into f \varphi_1(f)
\end{array}$$
which yields
$$
(\lambda - K \lambda_1 (f)) \leq 0\,,
$$
 a contradiction if $\lambda > K\lambda_1 (f)$. 
 \end{proof}

\begin{remark} 
Let us come back to problem \rife{main} with a general datum $f\in L^q (\Omega)$, $q>\frac{N}{2}$ 
and $a_m$ satisfying \rife{kappa}. 
The range of nonexistence proven  above can be explicitly characterized in terms of the first 
eigenvalue $\lambda_{1}(A,f)$. In fact, starting from the condition on $\lambda$ found at the 
end of the proof we can choose $\lambda=1$ to show that no solutions do exist if
 $$
 \lambda_1(A, f) <\frac{1}{K}.
 $$  
\end{remark}
 \bk 
 \section{Existence of a solution}\label{sec4}

Now we deal with our existence results.
Recall that we are dealing with problem 
\begin{equation}\label{eg}
\begin{cases}
\dys -\sum_{m=1}^{\infty}  \div (A_m (x) \nabla u^m)= f &\text{in}\ \Omega \\
u=0 & \text{on}\ \partial \Omega.
\end{cases}
\end{equation}

  The argument will be by approximation.  For $s\geq  0 $, we define 
\be\label{qn}
Q_n(s):= \sum_{m=1}^{n} a_m s^m, 
\ee
and we consider for $f\in L^q (\Omega) $, $q>\frac{N}{2}$ the unique weak solutions $v$ of the following problem
\begin{equation}\label{ng}
\begin{cases}
\dys -\div( A(x) \nabla v )= f &\text{in}\ \Omega \\
v=0 &\ \text{on}\  \partial\Omega\,;
\end{cases}
\end{equation} 
then we set $u_n=Q_{n}^{-1}(v)$ which is well defined since $Q_n$ is strictly increasing for $s>0$  and $Q_n(s)\to +\infty$ as $s$ goes to infinity. Moreover it is easy to check that $v\geq 0$, so that also $u_n\geq 0$.\bk
 
Therefore we have 
 \begin{proposition}\label{stime}
 Let $u_n=Q_{n}^{-1}(v)$ where $v$ is the weak solution to problem \rife{ng}. Then there exists $\overline{c}>0$ such that
 \begin{equation}\label{estim}
 \|u_n\|_{\huz}+\|u_n\|_{\lio}\leq \overline{c}\|f\|_{L^q (\Omega)}.
 \end{equation}
 
 In particular, there exists a nonnegative function $u\in\huz$ such that $u_n$ converges 
 to $u$ weakly in $\huz$ and $a.e.$ on $\Omega$, and $\|u\|_{\lio}\leq \sigma$, where $\sigma$ is defined by \rife{rc}. 
 \end{proposition}
 \begin{proof}
We take $u_n$ as test function in \rife{ng} and we use that $v=Q_n (u_n)$. So that we have
$$
\sum_{m=1}^{n} a_m \into A(x)\nabla u_n^m\cdot\nabla u_n =\into f u_n\,.
$$

 We drop al nonnegative terms but the first, and we get, using both H\"older and Sobolev inequalities
 $$
 \|u_n\|_{\huz}\leq \frac{C}{a_1}\|f\|_{L^q(\Omega)}.
 $$
Now we choose $G_k(u_n)$ as test function in the weak formulation of \rife{ng}. Dropping again all positive terms but the first, one obtain
$$
a_1\alpha \into |\nabla G_k (u_n) |^2\leq \into f G_k (u_n).
$$

So that, by a standard Stampacchia type argument (see \cite{s}) we readily have that there exists a positive $c$, such that
$$
\|u_n\|_{\lio}\leq c\|f\|_{L^q(\Omega)}\,,
$$
which completes the proof of estimate \rife{estim}. In particular there exists $u\in \huz$ such we can extract a  (not relabeled) subsequence $u_n$ convergent to $u$ both $a.e.$ on $\Omega$ and weakly in $\huz$.

 To obtain the bound with respect to $\sigma$ we start reasoning as in the proof of Proposition \ref{pro}. 
 Recall that all but a finite number of $a_m$ are different from zero, so, in the next argument, it will not be restrictive to suppose them \emph{all} different from zero. 

Hence,  we take $u_n$ as test function in \rife{ng} and we  drop all the positive terms but the $n$-th, to get, using the previous inequality  
$$
a_n n\into {u_n}^{n-1}|\nabla u_n|^2 \leq C \|f\|^2_{L^q (\Omega)}.
$$
On the other hand, using Poincar\'e inequality and then Chebyshev inequality, we have 
$$\begin{array}{l}\dys
\into {u_n}^{n-1}|\nabla u_n|^2=\frac{4}{(n+1)^2}\into |\nabla u_n^{\frac{n+1}{2}}|^2\\\\ \dys\geq c_p^p\frac{4}{(n+1)^2}\into |u_n|^{n+1}\geq c_p^p\frac{4 M^{n+1}}{(n+1)^2}\mis(\{u_n\geq M\})\,,\end{array}
$$
for a fixed number $M>\sigma$, and where $c_p$ is the Poincar\'e constant relative to $\Omega$.  Hence
$$
\mis(\{u_n\geq M\})\leq C\frac{(n+1)^2  \|f\|^2_{L^q (\Omega)}}{n a_n M^n},
$$
and, using the definition of $\sigma$, the right hand side behaves as $n(\frac{\sigma}{M})^n$, which implies, using that $\sigma<M$ and taking the liminf as $n$ goes to infinity, $u\leq M$, a.e. Due to the arbitrary choice of $M$ we get $\|u\|_{\lio}\leq \sigma$. 
\bk 
 
 \end{proof}

 \subsection{Existence for small data}
 
 First of all we want to prove that a solution does exist, for every $\{a_m\}$,  if the size of the datum is small enough no matter of the condition \rife{cond} is satisfied or not (cfr. with Theorem \ref{noex}). We have the following
 
 \begin{theorem}\label{piccolo}
 There exists $\overline{\lambda}$ such that, if $\|f\|_{L^q(\Omega)}< \overline{\lambda}$,
  then problem \rife{eg}  has a weak solution $u$ for every $\{a_m\}$ such that $\sigma>0$. 
  Moreover $\|u\|_{\lio}< \sigma$.
 \end{theorem}
\begin{proof}
Let $\sigma$ be defined as in \rife{rc}.  Thanks to Proposition \ref{stime} we know that  solutions to problem \rife{ng} satisfy 
$$
\|u_n\|_{\lio}\leq \overline{c}\|f\|_{L^q(\Omega)},
$$
so we choose $\overline{\lambda}$ as
$$
\overline{\lambda}=\frac{\sigma}{\overline{c}}.
$$

Using again Proposition \ref{stime} we know that  $u_n$ converges weakly towards a function 
$u$ in $\huz$ and a.e. on $\Omega$, moreover, thanks to the choice of $\overline{\lambda}$, we have  $\|u\|_{\lio}<\sigma $. In particular, for fixed $m$ we have
$$
\into A(x)  \nabla u_n^m\cdot\nabla \varphi\longrightarrow \into A(x)\nabla u^m\cdot\nabla \varphi\ \ \text{as} \ n\to+\infty,
$$
for any $\varphi\in\huz$. 

Moreover, we have
$$\begin{array}{l}
\dys\sum_{m=1}^{n}a_m  \into A(x)  \nabla u_{n}^{m} \cdot \nabla \varphi\\\\
\dys=\sum_{m=1}^{\infty}a_m m \chi_{\{1\leq m\leq n\}}(m) \into  u_n^{m-1} A(x) \nabla u_n \cdot \nabla \varphi,
\end{array}$$
Where, 
$$\begin{array}{l}
\dys\left| a_m m \chi_{\{1\leq m\leq n\}}(m) \into  u_n^{m-1} A(x) \nabla u_n \cdot \nabla \varphi\right|\\\\
\dys\leq C\beta a_m m \|u_n\|_{\lio}^{m-1}\leq C \beta a_m m \gamma^{m-1},
\end{array}$$
where $C$ only depends on $\overline{c}$ and on the norm of $\varphi$ in $\huz$, and   $\gamma<\sigma$. So we can apply  Lebesgue dominated convergence theorem with respect to the measure $\sum_{i=1}^{\infty}\delta_i$ 
 to pass to the limit  in \rife{ng} and we conclude.
\bk 
\end{proof}
\medskip

 As a consequence of this result we can better characterize the sets of data in which
  we have either existence or nonexistence. As before, to simplify the exposition, we shall consider 
  the following  problem depending on the parameter $\lambda$
 \begin{equation}\label{egl}
\begin{cases}
\dys -\sum_{m=1}^{\infty}  \div (A_m (x) \nabla u^m)= \lambda f &\text{in}\ \Omega \\
u=0 & \text{on}\ \partial \Omega\,,
\end{cases}
\end{equation}
with $A_m(x)$ and $f$ a fixed function in $L^q(\Omega)$, $q>\frac{N}{2}$.

Let us define the following 
$$
\Lambda_f:= \sup\{\lambda: \exists u \ \text{and}\ \gamma,\ \text{with $u$ solution of \rife{egl} }, \ u\leq 
\gamma<\sigma\}
$$
 Notice that $\Lambda_f$ is always strictly positive thanks to Theorem \ref{piccolo}. Now, if 
$$
\sum_{m=1}^{\infty}  a_m \sigma^m<+\infty. 
$$
then, by Theorem \ref{noex}, $\Lambda_f <+\infty$.   On the other hand, we will see in the 
 Section \ref{sub}  that, if the previous condition does not hold, then 
$\Lambda_f =+\infty$ (i.e. Theorem \ref{esiste}).  In any cases we can state the following 
result in which, as before,  the value $\sigma=+\infty$ is allowed.
\begin{theorem}
Let $\{a_m\}$ be such that $\sigma>0$. Then a solution to problem \rife{egl} does exist
 for any   $0<\lambda<\Lambda_f$. 
\end{theorem}

\begin{proof}
Let us fix $\lambda <\Lambda_f$.  We want to prove that a solution does exist. By definition there exists a solution $v$ of problem \rife{egl}  
with datum  $\mu\in (\lambda,\Lambda_f)$ with  $v\leq \gamma$.
We consider the following problem
\be\label{quel}
\begin{cases}
\dys -\div(A(x)\nabla  \overline{Q} (u) )= \lambda f &\text{in}\ \Omega \\
u=0 &\ \text{on}\  \partial\Omega\,, 
\end{cases}
\ee
where
$$\overline{Q}(s)=
\begin{cases}
\dys \sum_{m=1}^{\infty} a_m s^m &\text{if}\ s\leq \gamma \\
\dys  s \sum_{m=1}^{\infty} a_m \gamma^m &\ \text{if}\ s>\gamma. 
\end{cases}
$$
The solution $u$ does exist since  $\overline{Q}(s)$ is increasing and surjective on $\re^{+}$. 
Now, since $v\leq \gamma$, then $v$ turns out to solve problem \rife{quel} with  $\mu f$ as datum. 
Therefore, we can use  $(\overline{Q}(u)-\overline{Q}(v))^+$ as test function in the problems 
solved by $u$ and $v$ respectively, and then we subtract the second from the first one to obtain,
 using also \rife{a1}
$$
\alpha \int_{\{\overline{Q}(u)\geq \overline{Q}(v)\}} |\nabla (\overline{Q}(u)-\overline{Q}(v)) |^2\leq 0,
$$
that implies 
 $\overline{Q}(u)\leq \overline{Q}(v)$ a.e. on $\Omega$. Thus, since $\overline{Q}$ increases, 
 $u\leq v \leq \gamma$ and we conclude, since, thanks to the definition of $\overline{Q}$,
  this is equivalent  for $u$ to solve \rife{egl} .
\end{proof}

An interesting question related to the previous result is the following  
\begin{op}
What happens if $\lambda\to\Lambda_f $? Does a solution exist for problem \rife{egl} with datum $\Lambda_f f $ ?
In which sense?
\end{op}

\medskip

\subsection{Existence in the general case}\label{sub}

Here we complete the proof of Theorem  \ref{th} by giving the following

\begin{proof}[Proof of Theorem \ref{esiste}]

As before,  for $ s\geq 0$, we define
$$
Q_n(s):= \sum_{m=1}^{n} a_m s^m, 
$$
and we consider the weak solutions of  problem \rife{ng}
which exists, as before,  since $Q_n$ is strictly increasing for $s>0$ and surjective on $\re^+$. \bk

Using again a  Stampacchia type argument (see \cite{s}) we readily have 
$$
\|Q_n (u_n)\|_{\lio}\leq c\|f\|_{L^q(\Omega)}\,,
$$
for a fixed constant $c>0$. 
Moreover, thanks to Proposition \ref{stime} we know that 
 $$
 \|u_n\|_{\huz}\leq \overline{c}\|f\|_{L^q(\Omega)},
 $$
and there exists $u\in\huz$ such that $u_n$ weakly converges to $u$ in $\huz$, $u_n\to u $ a.e. on $\Omega$ with  $\|u\|_{\lio}\leq \sigma$.

To simplify the notation, let us define $\lambda\equiv\|f\|_{L^q(\Omega)}$. Moreover,
 since $Q_n$ is strictly increasing  we have
\begin{equation}\label{418}
\|u_n\|_{\lio}\leq Q^{-1}_n(c\lambda)\,.
\end{equation}

Our aim is to prove that, for fixed $f$ (and so $\lambda$),  there exists $0<\gamma<\sigma$ such that $Q_{n}^{-1}(c\lambda)\leq \gamma$ for $n$ large enough. In fact, it is enough to choose $\gamma$ such that
$$
\sum_{m=1}^{\infty}a_m \gamma^{m}> c\lambda\,,
$$
 that is possible by Abel's theorem since $a_m$ are nonnegative real numbers, $\forall m>0$. \bk 
Since $Q_{n}(\gamma)$ converges to $\sum_{m=1}^{\infty}a_m \gamma^{m}$, there exists $\overline{n}$ large enough such that $Q_{n}(\gamma) > c\lambda$ for every $n \geq \overline{n}$, and so
$Q_{n}^{-1}(c\lambda) < \gamma$ for $n \geq \overline{n}$.

Now, thanks to \rife{418}, it is easy to pass to the limit and to conclude as in the proof  of Theorem \ref{piccolo}.
\bk 
\end{proof}

\salta{
\begin{remark}
Another interesting (and open) question is the following: 
can we prove existence of a solution (even a distributional solution) if the assumptions on the
 regularity of the datum is weakened (e.g. $f\in \elle1$ or even a measure) ?  
\end{remark}
}

\section{Approximating sequences} 
In this last section  we want to give some further remarks on what happens to the
 approximating solutions when existence does not hold. 
 
Consider the approximating problems \rife{ng}, that is, once we define as before $Q_n (u_n)=v$ 
\begin{equation}\label{ega}
\begin{cases}
\dys -\sum_{m=1}^{n}  \div (A_m (x) \nabla u_n^m)= f &\text{in}\ \Omega \\
u_n=0 & \text{on}\ \partial \Omega\,.
\end{cases}
\end{equation}

As we already noticed, the  case $\sigma=0$ yields trivially nonexistence by Definition \ref{def}, since the only
 possible function satisfying \rife{defe} is $u=0$.  
We want to strenghten this fact by observing that the approximating sequences of 
solutions of problems \rife{ega} actually converges to zero as $n$ goes to infinity. The proof of this fact
 being already contained in the proof of Proposition \ref{stime}, since, for any $\delta>0$,  we got 
$$
\mis(\{u_n\geq \delta\})\leq \frac{C n }{\delta^n a_n}\to 0, 
$$
as $n\to+\infty$, an so 
$$
\mis(\{u\geq \delta\})=0, 
$$
for any $\delta>0$. That is $u=0$.  

\medskip 
So let us focus on the case $\sigma>0$ in which
\be\label{524}
\sum_{m=1}^\infty a_m \sigma^m<+\infty
\ee 
is in force.  As pointed out in  Example $3.1$ of \cite{o} (see also \cite{blop} where a finer analysis
concerning singular quasilinear equations is developed) some singular problems turn out to develop 
\emph{generalized solutions} obtained as limit of approximating sequences with \emph{flat zones} of positive measure 
in 
which they are constant and equal to $\sigma$ (where $\sigma$ is the singular point of the problem).
 For the sake of exposition, let us explain what happens with a constant datum $\lambda\in\re$. 
If $v$ solves
$$
\begin{cases}
 -\div (A(x) \nabla v)= \lambda &\text{in}\ \Omega \\
v=0 &\text{on}\ \partial \Omega\,,
\end{cases}
$$
 we consider, as before, the change of variable $v=Q_n (u_n)$ where $Q_n$ is defined as in \rife{qn}. 
It is always possible to choose $\lambda$ large enough such that 
$$
\mis(\{v>K\})>0, 
$$
where $K$ is as in \rife{kappa}. 
It is easy to see that, by definition, on the set $\{v>K\}$, $u_n$ converges a.e. to $\sigma$ and so
 the limit $u$ of the approximating solutions turns out to have a flat zone of positive measure.

\end{document}